\documentclass[11pt]{article}%
\usepackage{amsfonts}
\usepackage{amsmath}
\usepackage{color}
\usepackage{soul}
\usepackage{geometry}
\usepackage{amssymb}
\usepackage{graphicx}
\usepackage{enumitem}%
\providecommand{\U}[1]{\protect\rule{.1in}{.1in}}
\newtheorem{theorem}{Theorem}

\newtheorem{definition}[theorem]{Definition}

\newtheorem{lemma}[theorem]{Lemma}

\newtheorem{remark}[theorem]{Remark}

\newenvironment{proof}[1][Proof]{\noindent\textbf{#1.} }{\ \rule{0.5em}{0.5em}}

\begin{document}

\title{Mean field limit of interacting filaments for 3D Euler equations}

\author{Hakima Bessaih\thanks{University of Wyoming, Department of Mathematics, Dept.
3036, 1000 East University Avenue, Laramie WY 82071, United States,
bessaih@uwyo.edu}, Michele Coghi\thanks{ University of Bielefeld, Universit\"atsstra\ss e 25, 33615 Bielefeld, Germany, michele.coghi@gmail.com}, Franco
Flandoli\thanks{ Dipartimento di Matematica, Universit\`{a} di Pisa, Largo
Bruno Pontecorvo 5, 56127, Pisa, Italy, flandoli@dma.unipi.it}}
\date{}
\maketitle

\begin{abstract}
	The 3D Euler equations, precisely local smooth solutions of class $H^s$ with $s>5/2$, are obtained as a mean field limit of finite families of interacting curves, the so called vortex filaments, described by means of the concept of $1$-currents. This work is a continuation of \cite{BCF-2017}, where a preliminary result in this direction was obtained, with the true Euler equations replaced by a vector valued non linear PDE with a mollified Biot-Savart relation.	
\end{abstract}


\section{Introduction}

In our previous paper \cite{BCF-2017},  we investigated a mean-field type convergence
of a family of vortex lines (called vortex filaments) to a smoothed version of
the 3D Euler equations, where also the interaction between the vortex lines was
suitably mollified. This present paper is a continuation of \cite{BCF-2017} with the purpose of
proving convergence of the mollified vortex line dynamics to the
\textit{true} 3D Euler equations. The result is local in time, since we deal
with relatively smooth solutions of 3D Euler equations in their vorticity formulation, which exist and are
unique in Sobolev spaces $H^{s}$ for $s>\frac{3}{2}$. The passage from the
mollified to the true Euler equations was missing in our previous work since
we were unable to overcome a difficulty concerned with the local time interval
where the convergence should take place. It is quite clear that mollified and
true Euler equations have unique $H^{s}$ local solutions on a common local
time interval, as proved below in Theorem \ref{sobolev solutions euler}. However, it is not clear a
priori that the mollified vortex dynamics has a local solution on an interval
which is independent of the mollification. Under special assumptions on the
mollification kernel, however, one can prove globality of the vortex
dynamics:\ this is the key property used below to close the approximation
result. Let us also mention the much more difficult open problem of dealing
with the true vortex dynamics instead of the mollified one;\ however, this
problem is unsolved even for a single vortex filament, so the question of
convergence of a family of interacting filaments to Euler equations is
premature. 

The idea of a mean field of vortex filaments has been already
investigated with some success in  \cite{LionsMajda, Lions}. However, in those works the equation limit  is not
Euler equations, due to some idealizations in the vortex model. The filaments considered in these papers were assumed to be parallel. The present work is the first one that connects vortex filament models to the true Euler equations.

Our motivation comes from studying the Euler equations in $\mathbb{R}^3$

\begin{equation}\label{euler}
\begin{cases}
\partial_t v+(v \cdot \nabla) v +\nabla p =0 \\
\nabla \cdot v=0\\
v(0,x)=v_{0}(x)
\end{cases}
\end{equation}
where $v=v(t,x)$ denotes the velocity vector field, $x\in \mathbb{R}^3$, $p=p(t,x)$  the scalar pressure. If we denote by $\xi$ the vorticity field, $\xi=\nabla\times v$ then $\xi$ satisfies 

\begin{equation}\label{vorticity}
\begin{cases}
\partial_t \xi+(v \cdot \nabla) \xi = (\xi \cdot \nabla) v \\
\nabla \cdot \xi=0,\\
\xi(0,x)=\xi_{0}(x).
\end{cases}
\end{equation}
The vorticity and velocity fields $\xi$, $v$ are related by the Biot-Savart formula
\begin{equation*}
v(x)=(K\ast\xi)(x),
\end{equation*}
where  $K$ is the singular matrix  (at 0) given by
\begin{equation*}
K(x)=\frac{-\Gamma}{4\pi|x|^{3}}\left(
                         \begin{array}{ccc}
                           0 & x_{3} & -x_{2} \\
                           -x_{3} & 0& x_{1} \\
                           x_{2} & -x_{1} & 0 \\
                         \end{array}
                       \right).
\end{equation*}
and $\Gamma$ is the circulation associated to the velocity field $v$.
The Kernel $K$ can also be rewritten as follows:  for $x, h\in \mathbb{R}^3$
\begin{equation*}
K(x)h=\frac{\Gamma}{4\pi} \frac{x} {|x|^{3}}\times h.
\end{equation*}
Using the Biot-Savart formula, the vorticity equation \eqref{vorticity} becomes  
 \begin{equation}\label{Kvorticity}
\begin{cases}
\partial_t \xi+[(K\ast \xi) \cdot \nabla] \xi = (\xi \cdot \nabla) (K\ast\xi) \\
\nabla \cdot \xi=0,\\
\xi(0,x)=\xi_{0}(x)
\end{cases}
\end{equation}

As already mentioned, our aim is to prove a mean field result for equation \eqref{Kvorticity}, in a sense analogous to \cite{Dob} and \cite{MarPulv}, when interacting point particles are replaced by interacting curves; see the discussion about motivations in \cite{BCF-2017}.  We consider curves which vary in time and which are parametrized by $\sigma\in\left[  0,1\right]  $. Given
$t\geq0$, we consider the family of $N$ curves%
\[
\left\{  \gamma_{t}^{i,N,\delta}\left(  \sigma\right)  ;\sigma\in\left[  0,1\right]
,i=1,...,N\right\}
\]
in $\mathbb{R}^{d}$. For shortness we shall often write
$\left\{  \gamma_{t}^{i,N,\delta}\right\}$. 

These curves vary in time and interact through the equations
\begin{equation}
\frac{\partial}{\partial t}\gamma_{t}^{i,N,\delta}\left(  \sigma\right)  =\sum
_{j=1}^{N}\alpha_{j}^{N}\int_{0}^{1}K^\delta\left(  \gamma_{t}^{i,N,\delta}\left(
\sigma\right)  -\gamma_{t}^{j,N,\delta}\left(  \sigma^{\prime}\right)  \right)
\frac{\partial}{\partial\sigma^{\prime}}\gamma_{t}^{j,N,\delta}\left(  \sigma
^{\prime}\right)  d\sigma^{\prime} \label{interacting curves}%
\end{equation}
where $\alpha_{j}^{N}$ behaves as $\frac{1}{N}$ when the number of curves grows and  $K^\delta:\mathbb{R}^{3}\rightarrow \mathbb{R}^{3\times 3}$ is a smooth matrix-valued function (precisely, we need
$K^\delta$ of class $\mathcal{U}C_{b}^{3}(\mathbb{R}^{3},\mathbb{R}^{3\times 3})$, 
see \cite{BCF-2017}) which is defined as a mollification of the Biot-Savart kernel $K$. Notice that the curves $\gamma_t^{i, N,\delta}$ will also depend on $\delta$. Indeed, without the mollification we are unable to prove well-posedness of system \eqref{interacting curves}. 
The final limit to Euler equations is not uniform in $N$ and $\delta$: we have to choose $\delta$ depending on $N$, see \eqref{initial condition convergence}.

Let us briefly comment on previous investigations about vortex filaments dynamics. In our previous paper \cite{BCF-2017}, as a byproduct of our mean field result,  we were able to prove the well posedness of solutions of \eqref{interacting curves} which is a system of $N$ interacting  filaments. We proved that there exists unique global solutions in the space of smooth curves. Previous results  on the dynamic of one filament has been studied with some success in \cite{BerBes, BesGubRus, BesWij} where local in time solutions  were shown in some Sobolev spaces. Global solutions for the dynamic of one filament can be found in \cite{BerGub, BrzGubNek}. In all these papers, the study was focused on one filament  and sometimes with some stochastic features.   Numerical results can be found in \cite{Chorin, Leonard}.  A statistical approach was also developed using some Gibbs measures. This approach has been used with some success in \cite{LionsMajda} for $N$ filaments that are nearly parallel. For stochastic Brownian filaments, a similar approach was used in \cite{BesFla-2004}. The Gibbs ensemble for these stochastic Brownian filaments were first described by \cite{Flandoli} and \cite{FlaGub-2002}. 

As explained in \cite{BCF-2017}, we associate to such a family of curves a
distributional vector field (called a 1-current) defined as

\begin{equation}
\xi_{t}^{N,\delta}=\sum_{j=1}^{N}\alpha_{j}^{N}\int_{0}^{1}\delta_{\gamma_{t}%
^{j,N,\delta}\left(  \sigma\right)  }\frac{\partial}{\partial\sigma}\gamma_{t}%
^{j,N,\delta}\left(  \sigma\right)  d\sigma. \label{xi-N0}%
\end{equation} 
which plays the same role as the empirical measure in classical mean-field theory. This current also depends on $\delta$.

In \cite{BCF-2017} it was proved that $\xi_{t}^{N,\delta}$ converges in appropriate topologies to $\xi_t^\delta$, the
unique solution of a vector valued nonlinear PDE, of the form
\begin{equation}
\frac{\partial\xi_{t}^\delta}{\partial t}+\left[  \left(  K^\delta \ast\xi^\delta_{t}\right)
\cdot\nabla\right]  \xi^\delta_{t}=\xi^\delta_{t}\cdot\nabla\left(  K^\delta \ast\xi^\delta_{t}\right).
\label{PDE for currents}
\end{equation}
 We will use the result on the regularized equation to establish a mean field result for the 3D Euler equations in their vorticity formulation \eqref{Kvorticity}. 
 
 The main result of this paper is Theorem \ref{main_theorem}, which states that, given $\xi$ solution to \eqref{Kvorticity}, there exists a sequence of empirical measures $\xi_t^{N,\delta}$ of the form \eqref{xi-N0} which converges toward $\xi$. We base the proof of this theorem on the same ideas as the classical paper of Marchioro and Pulvirenti, \cite{MarPulv}. The convergence will be split in two parts. First, we proved in \cite{BCF-2017} for a fixed $\delta$ the convergence
 \begin{equation}\label{N convergence}
	 \xi^{N, \delta} \to \xi^{\delta}, \quad \mbox{as } N\to +\infty
 \end{equation}
 where $\xi^{N,\delta}$ is defined in \eqref{xi-N0} and $\xi^\delta$ is a solution to \eqref{PDE for currents}. 
 
 In what follows, we expand on that result, showing that there exists a {\it common time-interval}, independent of $\delta$, where there exist solutions of 
\eqref{Kvorticity} and \eqref{PDE for currents} and, where the convergence \eqref{N convergence} takes place.  The fundamental result to find this common interval is that \eqref{interacting curves} has a global in time solution for every $\delta$.
 The second crucial ingredient for the proof of Theorem \ref{main_theorem} is the convergence  \begin{equation}\label{delta convergence}
  \xi^\delta \to \xi, \quad \mbox{as } \delta \to 0.
 \end{equation}

 Hence, the convergence $\xi^{N, \delta} \to \xi$ can be split in \eqref{N convergence} and \eqref{delta convergence}. However, the convergence in \eqref{N convergence} is not uniform in $\delta$. We keep this into account in the proof of Theorem \ref{main_theorem} and we show that one can choose the regularization $\delta$ depending on $N$ such that the final convergence holds true.
 
 In Section \ref{section preliminaries} we briefly introduce the mathematical objects that we need in the remainder of the paper. 
In Section  \ref{section lagrangian current dynamics}, the main results of our previous paper \cite{BCF-2017} are recalled. Moreover, a collection of refined estimates are given in Lemma \ref{flow}. These estimates are necessary for the following section on mean field results for the {\it true} Euler equations \eqref{vorticity}.   We begin Section \ref{section convergence regularization} by recalling some known results on local classical solutions for the three dimensional Euler equations.  We then proceed in proving the stability of the Euler equations under the regularization of the Biot-Savart Kernel $K$, namely \eqref{delta convergence}. Finally, we state and prove our  main result in Theorem \ref{main_theorem}.  We prove the mean field results both for the regularized and for the {\it true} 3D Euler equations.

\section{Notations and definitions}\label{section preliminaries}

\subsection{Spaces of functions}
We will often refer to $\xi^N$ as a $1$-current. Currents of dimension 1 (called 1-currents here) are linear continuous mappings on the space $C_{c}^{\infty}\left(  \mathbb{R}^{d},\mathbb{R}%
^{d}\right)$ of smooth compact support vector fields of $\mathbb{R}^{d}$,
see for instance \cite{GiaModSou}, \cite{KraPar}.
For a more in depth discussion about currents we refer to \cite{BCF-2017}, especially Section 2 and the Appendix.

Given $k,d,m\in\mathbb{N}$, we denote by $C_{b}^{k}(\mathbb{R}^{d}%
,\mathbb{R}^{m})$ the space of all functions $f:\mathbb{R}^{d}\rightarrow
\mathbb{R}^{m}$ that are of class $C^{k}$, bounded together with their
derivatives of order up to $k$. By $\mathcal{U}C_{b}^{3}(\mathbb{R}%
^{d},\mathbb{R}^{m})$, we denote the subset of $C_{b}^{3}(\mathbb{R}%
^{d},\mathbb{R}^{m})$ of those functions $f$ such that $f$, $Df$ and $D^{2}f$
are also uniformly continuous.

We denote by $\mathcal{S}(\mathbb{R}^d)$ the Schwartz space of rapidly decreasing functions, defined as
\begin{equation*}
\mathcal{S}(\mathbb{R}^d) := \left\{ f\in C^\infty(\mathbb{R}^d) \; : \; \sup_{x\in \mathbb{R}^d} \vert x^\alpha f^{(\beta)}(x)\vert < +\infty \quad \forall \alpha, \beta \in \mathbb{Z}_{+}^d  \right\}
\end{equation*}

\subsection{Spaces of measures}
On the space $C_{b}\left(
\mathbb{R}^{3};\mathbb{R}^{3}\right) = C^0_{b}\left(
\mathbb{R}^{3};\mathbb{R}^{3}\right)$ of continuous and bounded vector
fields on $\mathbb{R}^{3}$, denote the uniform topology by $\left\Vert
\cdot\right\Vert _{\infty}$.
 Throughout the paper we shall always deal with
 the following Banach space of 1-currents:%
 \[
 \mathcal{M}:\mathcal{=}C_{b}\left(  \mathbb{R}^{3};\mathbb{R}^{3}\right)
 ^{\prime}.
 \]
 The topology induced by the duality will be denoted by $\left\vert
 \cdot\right\vert _{\mathcal{M}}$:%
 \[
 \left\vert \xi\right\vert _{\mathcal{M}}:=\sup_{\left\Vert \theta\right\Vert
 	_{\infty}\leq1}\left\vert \xi\left(  \theta\right)  \right\vert .
 \]
 As already seen in \cite{BCF-2017} to deal with approximation by filaments, it is essential to consider the weak topology
\[
\left\Vert \xi\right\Vert =\sup\{\xi(\theta)\;|\;\Vert\theta\Vert_{\infty
}+\text{Lip}(\theta)\leq1\}
\]
where Lip$(\theta)$ is the Lipschitz constant of $\theta$. We set%
\[
d\left(  \xi,\tilde{\xi}\right)  =\left\Vert \xi-\tilde{\xi}\right\Vert
\]
for all $\xi,\tilde{\xi}\in\mathcal{M}$. The number $\left\Vert
\xi\right\Vert $ is well defined and%
\[
\left\Vert \xi\right\Vert \leq\left\vert \xi\right\vert _{\mathcal{M}}%
\]
and $d\left(  \xi,\tilde{\xi}\right)  $ satisfies the conditions of a
distance.
 Convergence in the metric space $\left(  \mathcal{M},d\right)  $
 corresponds to weak convergence in $\mathcal{M}$ as dual to $C_{b}\left(
 \mathbb{R}^{3} ;\mathbb{R}^{3}\right)  $. 
 We shall denote by $\mathcal{M}_{w}$ the space $\mathcal{M}$ endowed by the
 metric $d$.
 The unit ball in $(\mathcal{M}, \vert\cdot\vert_{\mathcal{M}})$ is complete
 with respect to $d$, see \cite{BCF-2017} (Appendix).

\subsection{Push-Forward}
Throughout the paper we will deal with the notion of pull-back and push-forward with respect to a differentiable function. Let $\theta\in C_{b}\left(  \mathbb{R}^{3},\mathbb{R}^{3}\right)  $ be a
vector field (test function) and $\varphi:\mathbb{R}^{3}\rightarrow
\mathbb{R}^{3}$ be a map. When defined, the \textit{pull-back} of $\theta$ is%
\[
\left(  \varphi_{\sharp}\theta\right)  \left(  x\right)  =D\varphi\left(
x\right)  ^{T}\theta\left(  \varphi\left(  x\right)  \right)  .
\]
If $\varphi$ is of class $C^{1}\left(  \mathbb{R}^{3};\mathbb{R}^{3}\right)
$, then $\varphi_{\sharp}$ is a well defined bounded linear map from
$C_{b}\left(  \mathbb{R}^{3},\mathbb{R}^{3}\right)  $ to itself.

Given a $1$-current $\xi\in\mathcal{M}$ and a smooth map $\varphi:\mathbb{R}%
^{3}\rightarrow\mathbb{R}^{3}$, recall that the push-forward $\varphi_{\sharp
}\xi$ is defined as the current%
\[
\left(  \varphi_{\sharp}\xi\right)  \left(  \theta\right)  :=\xi\left(
\varphi_{\sharp}\theta\right)  ,\qquad\theta\in C_{b}\left(  \mathbb{R}%
^{3},\mathbb{R}^{3}\right)  .
\]

\section{Mollified Lagrangian dynamics}\label{section lagrangian current dynamics}

In the previous paper \cite{BCF-2017}, we analyzed the mean field of interacting filaments and the well posedness  for the PDE \eqref{PDE for currents} with initial condition $\xi_{0}\in\mathcal{M}$. These results are summarized in this section. Moreover,  some more accurate estimates on the Kernel and  the flow are computed in order to get the mean field limit for the 3D Euler vorticity equation \eqref{vorticity}.

In order to prove that the nonlinear vector-valued PDE
\eqref{PDE for currents}, with initial condition $\xi_{0}\in\mathcal{M}$, has
unique solutions in the space of currents, we adopt a Lagrangian point
of view:\ we examine the ordinary differential equation%
\begin{equation}
\frac{d}{dt} X_{t}=\left(  K^\delta \ast\xi_{t}\right)  \left(  X_{t}\right)  ,
\label{system eq 1}%
\end{equation}
consider the flow of diffeomorphisms $\varphi^{t,K^\delta \ast\xi}$ generated by it
and take the push forward of $\xi_{0}$ under this flow:
\begin{equation}
\xi_{t}=\varphi_{\sharp}^{t,K^\delta \ast\xi}\xi_{0},\qquad t\in\left[  0,T\right]  .
\label{system eq 2}%
\end{equation}
The pair of equations \eqref{system eq 1}-\eqref{system eq 2} defines a closed
system for $\left(  \xi_{t}\right)  _{t\in\left[  0,T\right]  }$.  We begin by giving the exact definition of the mollification $K^\delta$.

\subsection{The mollified Kernel}

Let us smooth the singular kernel $K$ by using a mollifier $\rho^{\delta}$. We denote by 
$K^{\delta}=\rho^\delta\ast K$ where $\rho^{\delta}=\delta^{-3}\rho(\frac{x}{\delta})$ where $\rho \in \mathcal{S}(\mathbb{R}^3)$. 
We assume that $\hat{\rho}$, the Fourier transform of $\rho$, has compact support and that
$\rho\geq 0$ and$ \int \rho=1$. As a consequence, the mollified kernel$K^\delta$ belongs to $C^\infty(\mathbb{R}^{3},\mathbb{R}^{3\times 3})$ and its Fourier transform has compact support. For the remainder of the paper this definition of $K^\delta$ will be in force. Let us mention, however, that the results of this section (see \cite{BCF-2017}) occur under  a less restrictive assumption, 
more precisely  under the assumption that  $K^\delta \in\mathcal{U}C_{b}^{3}(\mathbb{R}^{3},\mathbb{R}^{3\times 3})$.

If $\xi\in\mathcal{M}$ 
then $K^\delta \ast\xi$ is the vector
field in $\mathbb{R}^{3}$ with $i$-component given by%
\begin{equation*}
\left(  K^\delta \ast\xi\right)  _{i}\left(  x\right)  =\left(  K^\delta_{i\cdot}\ast
\xi\right)  \left(  x\right)  :=\xi\left(  K^\delta_{i\cdot}\left(  x-\cdot\right)
\right) 
\end{equation*}
where $K^\delta_{i\cdot}\left(  z\right)  $ is the vector $\left(  K^\delta_{ij}\left(
z\right)  \right)  _{j=1,...,3}$. We have
\[
\left\vert \left(  K^\delta\ast\xi\right)  \left(  x\right)  \right\vert
\leq\left\vert \xi\right\vert _{\mathcal{M}}\Vert K^\delta\Vert_{\infty}.
\]
Morover, we also have 
\[
\left\vert \left(  K^\delta\ast\xi\right)  \left(  x\right)  \right\vert
\leq\left\Vert \xi\right\Vert \left(  \Vert K^\delta\Vert_{\infty}+\Vert
DK^\delta\Vert_{\infty}\right)  .
\]

\subsection{Main results for the mollified problem}

In \cite{BCF-2017} the mean field limit to vortex filaments system \eqref{interacting curves}-\eqref{xi-N0} was analyzed. In this section we recall the main results from that paper.

\begin{theorem}
\label{thm on flow eq}i)\ (Maximal solutions) For every $\xi_{0}%
\in\mathcal{M}$, there is a unique maximal solution $\xi$ of the flow equations
\eqref{system eq 1}-\eqref{system eq 2} in $C\left(  [0,T_{\xi_{0}});\mathcal{M}\right)  $.

iii)\ (Continuous dependence) If $\xi_{0}^{n}\rightarrow\xi_{0}$ in
$\mathcal{M}_{w}$ and $\left[  0,T_0\right]  $ is a common time interval of
existence and uniqueness for the flow equations \eqref{system eq 1}-\eqref{system eq 2} with initial conditions $\xi^N_0$ and $\xi_0$, then for the
corresponding solutions $\xi^{n}$ and $\xi$ we have $\xi^{n}\rightarrow\xi$ in
$C\left(  \left[  0,T_0\right]  ;\mathcal{M}_{w}\right)  $.
\end{theorem}

There is an equivalence between the Lagrangian formulation and the Eulerian formulation that is:

\begin{lemma}
\label{lemma equivalence}A function $\xi\in C([0,T];\mathcal{M})$ is a
current-valued solution for the PDE \eqref{PDE for currents} if and only if it is given
by%
\[
\xi_{t}=\varphi_{\sharp}^{t,K^\delta\ast\xi}\xi_{0}.
\]

\end{lemma}
The solution $\xi$ has to be understood in the following sense:

\begin{definition}
\label{solution} We say that $\xi\in C([0,T];\mathcal{M})$ is a current-valued
solution for the PDE  \eqref{PDE for currents}  if for every $\theta\in
C_{b}^{1}(\mathbb{R}^{3};\mathbb{R}^{3})$ and every $t\in\lbrack0,T]$, it
satisfies%
\begin{equation*}
\xi_{t}\left(  \theta\right)  -\int_{0}^{t}\xi_{s}\left(  D\theta\cdot
(K^\delta \ast\xi_{s})  \right)  ds=\xi_{0}\left(  \theta\right)  +\int%
_{0}^{t}\xi_{s}\left(  (DK^\delta\ast \xi_{s})^{T}\cdot\theta\right)  ds.
\end{equation*}

\end{definition}

Notice that the previous is the weak formulation of \eqref{PDE for currents}, since $K^\delta$ is divergence free.

\begin{theorem}\label{global} (Global solutions) Let $\xi_0 \in \mathcal{M}$ and assume that $\xi$ is the maximal solution for the flow equations \eqref{system eq 1} - \eqref{system eq 2} as given by Theorem \ref{thm on flow eq}.  If $\xi_0$ has a compact support,  then for every $t\in [0,T]$
\begin{equation}\label{EE1}
\Vert K^\delta \ast\xi_t\Vert_{L^2}= \Vert K^\delta \ast\xi_0\Vert_{L^2}
\end{equation}

and 

\begin{equation}\label{EE2}
\Vert K^\delta \ast\xi_t\Vert_{\infty}\leq \Vert K^\delta \ast\xi_0\Vert_{L^2}
\end{equation}

\begin{equation*}\label{EE3}
\Vert \xi_t\Vert \leq |\xi_0|_{\mathcal{M}} e^{T\Vert K^\delta \ast\xi_0\Vert_{L^2} }
\end{equation*}

As a consequence, we deduce that  the solutions of \eqref{system eq 2} are global in time. 
\end{theorem}

\begin{remark}
	For the general formulation of Theorem \ref{thm on flow eq} and Theorem \ref{global} the reader can refer to \cite{BCF-2017}. We would like to stress that it is necessary in Theorem \ref{global} that the Fourier transform of $K^\delta$ has compact support and $\operatorname{div}(K^\delta) = 0$.
	\end{remark}

\subsection{Properties of the flow of diffeomorphism }\label{section the flow}

Let us denote by $C_{\delta}^{m}:=\Vert D^mK^\delta\Vert_{\infty}+\Vert D^{m+1}K^\delta\Vert_{\infty}  \approx \frac{1}{\delta^{4+m}}$.  Using the properties of the Kernel $K^\delta$ and the estimates \eqref{EE1} and \eqref{EE2}, It is not difficult to deduce that the following properties hold for every $\xi_t,\tilde{\xi}_t\in\mathcal{M}$%

\begin{equation}
\Vert K^\delta \ast\xi_t\Vert_{\infty}\leq C^0_{\delta}\Vert\xi_0\Vert
\label{B1-1}
\end{equation}

\begin{equation}
\Vert D K^\delta \ast\xi_t\Vert_{\infty} \leq C^1_{\delta}\Vert\xi_0\Vert
\label{B1-2}
\end{equation}

\begin{equation}
\Vert D^2K^\delta \ast\xi_t\Vert_{\infty}\leq C^2_{\delta}\Vert\xi_0\Vert
\label{B1-3}
\end{equation}

\begin{equation}
\label{B2}\|K^\delta \ast\xi_t- K^\delta \ast\tilde{\xi}_t \|_{\infty}\leq C^1_{\delta}\|\xi_t-\tilde{\xi}_t\|
\end{equation}

\begin{equation}
\Vert DK^\delta \ast\xi_t-D K^\delta \ast\tilde{\xi}_t \Vert_{\infty}\leq C^2_{\delta}\Vert\xi_t-\tilde{\xi}_t
\Vert\label{B3}
\end{equation}
The flow $\varphi^{t,K^\delta \ast\xi}$ associated to the ODE \eqref{system eq 1} satisfies the following properties.


\begin{lemma}
\label{flow} If $\xi\in C\left(  \left[  0,T\right]  ;\mathcal{M}_{w}\right)
$, then the flow $\varphi^{t,K^\delta \ast\xi}:\mathbb{R}^{3}\rightarrow\mathbb{R}^{3}$
is twice differentiable and satisfies, for all $t\in\left[  0,T\right]  $,

\begin{equation}
\left\Vert D\varphi^{t,K^\delta \ast\xi}\right\Vert _{\infty}\leq
e^{C^1_{\delta}T \Vert \xi_0\Vert} \label{Dphi}%
\end{equation}

\begin{equation}
\left\Vert \varphi^{t,K^\delta \ast\xi}-\varphi^{t,K^\delta \ast\tilde{\xi}} \right\Vert _{\infty}\leq C^0_{\delta}C^1_{\delta} \|\xi_0\| 
e^{C^1_{\delta} \Vert \xi_0\Vert}\int_{0}^{T}\left\Vert \xi_s-\tilde{\xi}_s\right\Vert ds \label{Lip-phi}%
\end{equation}
and for every $x\in\mathbb{R}^{3}$%
\[
\left\Vert D\varphi^{t,K^\delta \ast\xi}-D\varphi^{t,K^\delta \ast\tilde{\xi}} \right\Vert _{\infty}%
\]%
\begin{equation}
\leq C^2_{\delta}\|\xi_0\| e^{C^1_{\delta}T  (\left\Vert \xi_0\right\Vert
+\left\Vert \tilde{\xi}_0\right\Vert )}\left(  1+C^0_{\delta} C^1_{\delta}T \left\Vert
\xi_0\right\Vert e^{C^1_{\delta}T\left\Vert \xi_0\right\Vert }\right)
\int_{0}^{T}\Vert\xi_s-\tilde{\xi}\Vert ds. \label{Lip-Dphi}%
\end{equation}

Moreover, for every $x,y\in\mathbb{R}^{3}$,
\begin{equation}
|D\varphi^{t,K^\delta \ast\xi}\left(  x\right)  -D\varphi^{t,K^\delta \ast\xi}\left(  y\right)| \leq 
C^2_{\delta}T \left\Vert \xi_0\right\Vert e^{C^1_{\delta}\left(  2T\right) \left\Vert \xi_0\right\Vert }\left\vert x-y\right\vert . 
\label{Dphi-Lip}%
\end{equation}

\end{lemma}

\begin{proof} See \cite{BCF-2017} for a proof  in a more general setting but for the sake of completeness, we will perform the estimates in  details. 

Since the kernel $K^\delta$ is regular, then the flow map  $\varphi^{t,K^\delta \ast\xi}$ is differentiable and we have that
$\frac{d}{dt}D\varphi^{t,K^\delta \ast\xi}\left(  x\right)
=D(K^\delta \ast\xi) \left(  \varphi^{t,K^\delta \ast\xi}\left(
x\right)  \right)  D\varphi^{t, K^\delta \ast\xi}\left(  x\right)  $.
Hence,

\[
\left\vert D\varphi^{t,K^\delta \ast\xi}\left(  x\right)  \right\vert
\leq e^{\int_{0}^{t}\left\vert DK^\delta \ast\xi_{s} \left(
\varphi^{s,K^\delta \ast\xi }\left(  x\right)  \right)  \right\vert ds}%
\]%
\[
\left\Vert D\varphi^{t,K^\delta \ast\xi }\right\Vert _{\infty}\leq
e^{\int_{0}^{t}\left\Vert DK^\delta \ast\xi_{s}  \right\Vert _{\infty}ds}.
\]
Now, using the assumption \eqref{B1-2} we get \eqref{Dphi}.

For the estimate \eqref{Lip-phi}, notice that
\begin{align*}
\frac{d}{dt}\left(  \varphi^{t,K^\delta \ast\xi  }\left(  x\right)
-\varphi^{t,K^\delta \ast\tilde{\xi}  }\left(  x\right)  \right)   &
=K^\delta \ast\xi  \left(  \varphi^{t,K^\delta \ast\xi  }\left(
x\right)  \right)  -K^\delta \ast\xi ^{\prime}  \left(  \varphi
^{t,K^\delta \ast\tilde{\xi}}\left(  x\right)  \right) \\
&  =K^\delta \ast\xi   \left(  \varphi^{t, K^\delta \ast\xi }\left(
x\right)  \right)  -K^\delta \ast\xi  \left(  \varphi^{t,K^\delta \ast\tilde{\xi}}\left(  x\right)  \right) \\
&  +K^\delta \ast\xi \left(  \varphi^{t,K^\delta \ast\tilde{\xi}}\left(  x\right)  \right)  
-K^\delta \ast\tilde{\xi} \left(
\varphi^{t,K^\delta \ast\tilde{\xi} }\left(  x\right)  \right)
\end{align*}
hence%
\begin{align*}
\left\vert \varphi^{t,K^\delta \ast\xi }\left(  x\right)  -\varphi
^{t,K^\delta \ast\tilde{\xi} }\left(  x\right)  \right\vert  &  \leq
\int_{0}^{t}\left\Vert D K^\delta \ast\xi_{s} \right\Vert _{\infty
}\left\vert \varphi^{s,K^\delta \ast\xi  }\left(  x\right)  -\varphi
^{s,K^\delta \ast\tilde{\xi} }\left(  x\right)  \right\vert ds\\
&  +\int_{0}^{t}\left\Vert K^\delta \ast\xi_{s}  -K^\delta \ast\xi_{s}^{\prime} \right\Vert _{\infty}ds.
\end{align*}
Thus, using Gronwall's Lemma we get that%
\[
\left\vert \varphi^{t,K^\delta \ast\xi }\left(  x\right)  -\varphi
^{t,K^\delta \ast\tilde{\xi} }\left(  x\right)  \right\vert \leq\int%
_{0}^{t}\left\Vert K^\delta \ast\xi_{s}  -K^\delta \ast\xi_{s}^{\prime}\right\Vert _{\infty}
e^{\int_{s}^{t}\left\Vert D K^\delta \ast\xi_{r}
\right\Vert _{\infty}dr}ds.
\]
Now, using again assumptions \eqref{B1-2} and \eqref{B2}, we deduce \eqref{Lip-phi}.

In order to prove \eqref{Lip-Dphi}. Let us notice that
\begin{align*}
&  \frac{d}{dt}\left(  D\varphi^{t,K^\delta \ast\xi }\left(  x\right)
-D\varphi^{t,K^\delta \ast\tilde{\xi} }\left(  x\right)  \right) \\
&  =DK^\delta \ast\xi_{t}  \left(  \varphi^{t,K^\delta \ast\xi  }\left(
x\right)  \right)  D\varphi^{t, K^\delta \ast\xi  }\left(  x\right)
-D K^\delta \ast\xi_{t}^{\prime}\left(  \varphi^{t,K^\delta \ast\tilde{\xi} }\left(  x\right)  \right)  
D\varphi^{t,K^\delta \ast\tilde{\xi}  }\left(  x\right) \\
&  =DK^\delta \ast\xi_{t}  \left(  \varphi^{t,K^\delta \ast\xi}\left(
x\right)  \right)  D\varphi^{t, K^\delta \ast\xi}\left(  x\right)
-DK^\delta \ast\xi_{t} \left(  \varphi^{t,K^\delta \ast\xi }\left(
x\right)  \right)  D\varphi^{t, K^\delta \ast\tilde{\xi}  }\left(  x\right)
\\
&  +DK^\delta \ast\xi_{t} \left(  \varphi^{t,K^\delta \ast\xi}\left(
x\right)  \right)  D\varphi^{t,K^\delta \ast\tilde{\xi}  }\left(
x\right)  -DK^\delta \ast\tilde{\xi}_t  \left(  \varphi^{t,K^\delta \ast\tilde{\xi} }\left(  x\right)  \right)  
D\varphi^{t,K^\delta \ast\tilde{\xi} }\left(  x\right)  .
\end{align*}

Now, following very similar arguments to the ones used above, and by using the estimates \eqref{B1-1}, \eqref{B3}, \eqref{Dphi} and \eqref{Lip-phi}, we get  \eqref{Lip-Dphi}.

It is left to prove \eqref{Dphi-Lip}.
\begin{align*}
&|D\varphi^{t, K^\delta \ast\xi}\left(  x\right)     -D\varphi^{t,K^\delta \ast\xi}\left(  y\right)  |\\
\leq &  \int_{0}^{t}\left\vert D K^\delta \ast\xi_{s}(\varphi^{s,K^\delta \ast\xi}(x))D\varphi
^{t,K^\delta \ast\xi }(x)-D K^\delta \ast\xi_{s})(\varphi^{s,K^\delta \ast\xi}(y))D\varphi^{t,K^\delta \ast\xi}\left(
y\right)  \right\vert ds\\
\leq &  \int_{0}^{t}\left\vert D K^\delta \ast\xi_{s}(\varphi^{s,K^\delta \ast\xi}(x))D\varphi
^{s,K^\delta \ast\xi}(x)-D K^\delta \ast\xi_{s}(\varphi^{s,K^\delta \ast\xi}(x))D\varphi^{s,K^\delta \ast\xi}\left(
y\right)  \right\vert \\
&  +\left\vert D K^\delta \ast\xi_{s}(\varphi^{s, K^\delta \ast\xi}(x))D\varphi^{s,K^\delta \ast\xi }%
(y)-DK^\delta \ast\xi_{s}(\varphi^{s, K^\delta \ast\xi}(y))D\varphi^{s,K^\delta \ast\xi}\left(  y\right)
\right\vert ds\\
\leq &  \sup_{s\in\lbrack0,t]}\Vert D K^\delta \ast\xi_{s}\Vert_{\infty}\int_{0}%
^{t}\left\vert D\varphi^{s, K^\delta \ast\xi}\left(  x\right)
-D\varphi^{s,K^\delta \ast\xi}\left(  y\right)  \right\vert ds\\
&  +t\sup_{s\in\lbrack0,t]}\left(  \left\Vert D\varphi^{s, K^\delta \ast\xi}\right\Vert
_{\infty}^{2}\left\Vert D^{2} K^\delta \ast\xi_{s}\right\Vert _{\infty}\right)
\left\vert x-y\right\vert .
\end{align*}
We now apply Gronwall's Lemma and we get
\[
|D\varphi^{t,K^\delta \ast\xi }\left(  x\right)  -D\varphi^{t,K^\delta \ast\xi }\left(  y\right)  |\leq T\sup_{s\in\lbrack0,T]}\left(  \left\Vert
D\varphi^{s,K^\delta \ast\xi}\right\Vert _{\infty}^{2}\left\Vert D^{2} K^\delta \ast\xi_{s}\right\Vert _{\infty}\right)  
e^{T\sup_{s\in\lbrack0,T]}\Vert D K^\delta \ast\xi_{s}\Vert_{\infty}}\left\vert x-y\right\vert .
\]
Now, using  \eqref{B1-2}, \eqref{B1-3}  and \eqref{Dphi} we get \eqref{Dphi-Lip}.

\end{proof}

\section{Mean field result for the 3D Euler equations}\label{section convergence regularization}

Now, we are able to prove a mean field result for the 3D Euler equations. We have to do it in two steps. First we mollify the 3D Euler equations through the mollification of the Kernel $K$. We approximate the vorticity $\xi$ by the mollified one $\xi^\delta$ and we use the mean field result for the vorticity $\xi^\delta$ proved in \cite{BCF-2017} and we combine the two pieces.  In what follows, $H^s\left(  \mathbb{R}^{3}: \mathbb{R}^3\right) $  is the Sobolev space of vector valued functions, we denote by $\|\cdot\|_{H^s}$ the associated norm. In particular for $s=0$, $H^0\left(  \mathbb{R}^{3}, \mathbb{R}^3\right) $ reduces to the Lebesgue space $L^2\left(  \mathbb{R}^{3}, \mathbb{R}^3\right) $.  
\subsection{Local posedness for the Euler equations}

It is known that \eqref{euler} has local (in time) classical solutions:

\begin{theorem} Let $v_{0}\in H^m\left(  \mathbb{R}^{3}, \mathbb{R}^{3}\right) $ with $m>5/2$ such that ${\rm div}\ v_{0}=0$. Then, there exists $T_{0}>0$ and there exists a unique $v\in C([0,T]; H^m\left(  \mathbb{R}^{3}, \mathbb{R}^{3}\right) )$ solution of \eqref{euler}. 
\end{theorem}

The proof of this theorem can be performed through a fixed point argument and can be found for example in
\cite{Bourguignon} and also in \cite{BerMaj}. 
We will show here only how to get the estimate of the vorticity field in $H^{m-1}\left(  \mathbb{R}^{3}, \mathbb{R}^{3}\right) $. 
We define the linear operator $\Lambda=(-\Delta)^{1/2}$ and its powers $\Lambda^s$.
Hence $\Lambda^2=-\Delta$.
Note, in particular that $\Lambda^{s} $ maps 
$H^{r}\left(  \mathbb{R}^{3}, \mathbb{R}^{3}\right) $ into $H^{r-s}\left(  \mathbb{R}^{3}, \mathbb{R}^{3}\right) $.

For any $u,v,w\in H^1\left(  \mathbb{R}^{3}, \mathbb{R}^{3}\right) $ with $\nabla\cdot u=0$ we have
\begin{equation}\label{tril1}
\langle u \cdot \nabla v, w\rangle =-\langle u \cdot \nabla w, v \rangle, \qquad \langle u \cdot\nabla v, v\rangle =0.
\end{equation}
Define the commutator 
\[
[\Lambda^s,f]g= \Lambda^s(fg) - f \ \Lambda^s g .
\]
From \cite{KatoPonce} we have the following two lemmas.

\begin{lemma}[Commutator lemma]\label{comm}
Let $s>0$, $1<p<\infty$  and $p_2,p_3 \in (1,\infty)$ be such that
\[
\frac 1p\ge  \frac 1{p_1}+\frac 1{p_2},\qquad
\frac 1p\ge \frac 1{p_3}+\frac 1{p_4}.
\]
Then
\[
\|[\Lambda^s,f]g\|_{L_p}\le C\left(
\|\nabla f\|_{L_{p_1}}\|\Lambda^{s-1}g\|_{L_{p_2}}+
\|\Lambda^s f\|_{L_{p_3}}\|g\|_{L_{p_4}}
\right) .
\]
\end{lemma}
and 

\begin{lemma}\label{prod}
Let $s>\frac{3}{2}$ then $H^{s}(\mathbb{R}^3, \mathbb{R}^{3})$ is a Banach algebra, that is there is a $c>0$ such that for all 
$f, g\in H^{s}(\mathbb{R}^3)$
\[
\|\Lambda^s(fg)\|_{H^0} \le C \|f\|_{H^s} \| g\|_{H^s}
\]
\end{lemma}
In particular for $u,v \in \mathbb R^3$,
\[
[\Lambda^s,u]\cdot \nabla v
=
\Lambda^s\big((u \cdot \nabla)v\big)
-(u \cdot \nabla) \Lambda^s v .
\]
Hence,
\begin{equation}\label{tril-con-comm2}
\begin{split}
\langle \Lambda^s\big((u \cdot \nabla)v \big), \Lambda^s v\rangle &=
\langle [\Lambda^s,u]  \cdot \nabla v, \Lambda^s v\rangle +
\underbrace{\langle (u \cdot \nabla)\Lambda^s v , \Lambda^s v\rangle }
_{=0\text { by } \eqref{tril1}}
\\
&=
\langle [\Lambda^s,u] \cdot \nabla  v, \Lambda^s v\rangle .
\end{split}\end{equation}

Now the estimate on the vorticity field in $H^{m-1}(\mathbb{R}^3, \mathbb{R}^3)$ can be proved as in the proof of Theorem \ref{sobolev solutions euler} below.
	
\subsection{The mollified 3D Euler equations}
We are interested in the 3D Euler equations in their vorticity formulation. 
Consider the vector field $\xi^{\delta}$ solution of

\begin{equation}\label{vorticity-delta}
\begin{cases}
\partial_t \xi^{\delta}+(v^{\delta} \cdot \nabla) \xi^{\delta} = (\xi^{\delta} \cdot \nabla) v^{\delta} \\
\nabla \cdot \xi^{\delta}=0,\\
\xi^{\delta}(0,x)=\xi_{0}^{\delta}(x)
\end{cases}
\end{equation}
where 
\begin{equation}\label{v-K}
v^{\delta}(t, x)=(K^{\delta}\ast \xi^{\delta})(t, x),
\end{equation}

\begin{theorem}\label{sobolev solutions euler} Let $\xi^{\delta}_{0}\in H^s\left(  \mathbb{R}^{3}, \mathbb{R}^{3}\right) $ with $s>3/2$. Then, there exists $T_{0}>0$, independent of $\delta$ and there exists a unique $\xi^{\delta}\in C([0,T]; H^s\left(  \mathbb{R}^{3}, \mathbb{R}^{3}\right) )$ solution of \eqref{vorticity-delta}. 
\end{theorem}
\begin{proof} Here we only prove that there exist $T_0, C>0$, independent of $\delta$, such that $\sup_{0\leq t\leq T_0}\|\xi^{\delta}(t)\|_{H^s}\leq C$.

Let us apply the operator $\Lambda^{s}$ to \eqref{vorticity-delta} and then multiply by $\Lambda^{s} \xi^{\delta}(t)$. By using 
\eqref{tril-con-comm2} we obtain that

\begin{align*}
\frac{1}{2}\frac{d}{dt}\|\Lambda^s\xi^{\delta}(t)\|^2_{H^0}&\leq 
|\langle [\Lambda^s,v^{\delta}(t)]  \cdot \nabla \xi^{\delta}(t), \Lambda^s \xi^{\delta}(t)\rangle |+
|\langle \Lambda^s(\xi^{\delta}(t) \cdot \nabla v^{\delta}(t)), \Lambda^s \xi^{\delta}(t)\rangle |\\
&\leq  \|[\Lambda^s,v^{\delta}(t)]  \cdot \nabla \xi^{\delta}(t)\|_{H^0} \|\Lambda^s\xi^{\delta}(t)\|_{H^0}+
\|\Lambda^s(\xi^{\delta} (t)\cdot \nabla v^{\delta}(t))\|_{H^0} \|\Lambda^s\xi^{\delta}(t)\|_{H^0}
\end{align*}
By using Lemma \ref{comm} with $p=p_2=2$, $p_1=\infty$ and $p_3=6$ and $p_4=3$  and also Lemma \ref{prod}, we obtain
\begin{align*}
\frac{1}{2}\frac{d}{dt}\|\Lambda^s\xi^{\delta}(t)\|^2_{H^0} &\leq \left (\| \nabla v^{\delta}\|_{L^\infty} \|\Lambda^s\xi^{\delta}(t)\|_{H^0} + \|\Lambda^s v^{\delta}\|_{L^{6}}\|\Lambda \xi^{\delta}(t)\|_{L^3}\right)
 \|\Lambda^s\xi^{\delta}(t)\|_{H^0}\\
& +\|\Lambda^{s+1}v^{\delta}(t)\|_{H^0}\|\Lambda^s\xi^{\delta}(t)\|^2_{H^0}
 \end{align*}
Since $s>\frac{3}{2}$, $H^{s}\left( \mathbb{R}^{3}, \mathbb{R}^{3}\right) \subset L^\infty\left(  \mathbb{R}^{3}, \mathbb{R}^{3}\right) $, 
$H^1\left(  \mathbb{R}^{3}, \mathbb{R}^{3}\right) \subset L^6\left(  \mathbb{R}^{3}, \mathbb{R}^{3}\right) $ and 
$H^{1/2}\left(  \mathbb{R}^{3}, \mathbb{R}^{3}\right) \subset L^3\left(  \mathbb{R}^{3}, \mathbb{R}^{3}\right) $,
 then up to a constant that we omit, we have
$\|\Lambda \xi^{\delta}(t)\|_{L^3}\leq \left\Vert \Lambda\xi^{\delta}(t)\right\Vert _{H^{1/2}}\leq\left\Vert \Lambda^{s}%
\xi^{\delta}(t)\right\Vert _{0}$, and $\left\Vert \nabla v^{\delta}(t)\right\Vert
_{L^\infty}\leq\left\Vert \Lambda^{s+1}v^{\delta}(t)\right\Vert _{H^0}$; moreover,
$\left\Vert \Lambda^{s}v^{\delta}(t)\right\Vert _{L^6}\leq\left\Vert \Lambda
^{s+1}v^{\delta}(t)\right\Vert _{0}$. Hence the previous sum of terms is bounded
above (up to a constant) by%
\[
\frac{1}{2}\frac{d}{dt}\|\Lambda^s\xi^{\delta}(t)\|^2_{H^0}\leq C\left\Vert \Lambda^{s+1}v^{\delta}(t)\right\Vert _{H^0}\left\Vert \Lambda
^{s}\xi^{\delta}(t)\right\Vert _{H^0}^{2}.
\]
Now using the identity \eqref{v-K}, the definition of $K^\delta$, the properties of the convolution and the definition of $\rho$, we have
\begin{align}
\|\Lambda^{s+1}v^{\delta}(t)\|_{H^0}&=\|\Lambda^{s+1}(K^{\delta}\ast \xi^{\delta})(t)\|_{H^0}\nonumber\\
&=\|\Lambda^{s+1}(\rho^\delta\ast K\ast \xi^{\delta})(t)\|_{H^0}\nonumber\\
&\leq \|\rho^\delta\ast \Lambda^{s+1}(K\ast \xi^{\delta})(t)\|_{H^0}\label{s plus one of v}\\
&\leq  \|\Lambda^{s+1}(K\ast \xi^{\delta}(t))\|_{H^0}\nonumber\\
&\leq C\|\Lambda^{s}\xi^{\delta}(t)\|_{H^0}\nonumber
\end{align}
Let us mention that in the above inequality, we use that the linear map
\[
\xi\mapsto K\ast\xi
\]%
is continuous  from $H^{s}(\mathbb{R}^3, \mathbb{R}^3)$ into  $H^{s+1}(\mathbb{R}^3, \mathbb{R}^3)$ and that the linear map 
\[
g\mapsto\rho^{\delta}\ast g
\]
is equibounded  in $\delta$ from $H^{s+1}(\mathbb{R}^3, \mathbb{R}^3)$ into $H^{s+1}(\mathbb{R}^3, \mathbb{R}^3)$.

Plugging inequality \eqref{s plus one of v} in the previous estimate we get that

\begin{align*}
\frac{d}{dt}\|\Lambda^s\xi^{\delta}(t)\|^2_{H^0}&\leq C \|\Lambda^s\xi^{\delta}(t)\|^3_{H^0},
\end{align*}
which implies that 
\begin{align*}
\frac{d}{dt}\|\Lambda^s\xi^{\delta}(t)\|_{H^0}&\leq C\|\Lambda^s\xi^{\delta}(t)\|^2_{H^0}.
\end{align*}
then, after integrating

\begin{align*}
\frac{1}{\|\Lambda^s\xi^{\delta}(0)\|_{H^0}}-\frac{1}{\|\Lambda^s\xi^{\delta}(t)\|_{H^0}}\leq Ct
\end{align*}
Hence, if we assume that $T_{0}<\frac{1}{C\|\xi^{\delta}(0)\|_{H^s}}$, then we have the estimate

$$\sup_{0\leq t\leq T_{0}}\|\xi^{\delta}(t)\|_{H^s}^2\leq C. $$

\end{proof}

\subsection{Stability for the regularized solution}
Now we are able to state the convergence of $\xi^\delta$ to $\xi$. We have from the previous section that

\[\sup_{t\in\left[  0,T_{0}\right]  }\left\Vert \xi^{\delta}\left(  t\right)
\right\Vert _{H^{s}}\leq C
\]

for every $\delta\in\lbrack0,1)$.


\begin{theorem} \label{delta-cv} If  $\xi^{\delta}(0)  \rightarrow\xi (0)  $ in
$L^{2}\left(  \mathbb{R}^{3}, \mathbb{R}^3\right)  $ then $\xi^{\delta}\rightarrow\xi$
in $C\left(  \left[  0,T_{0}\right]  ;L^{2}\left(  \mathbb{R}^{3}, \mathbb{R}^3\right)
\right) $. 
\end{theorem}

\begin{proof}
We have that
\[
\partial_{t}\left(  \xi^{\delta}-\xi\right)  +\left(  v^{\delta}%
\cdot\nabla\right)  \xi^{\delta}-\left(  v\cdot\nabla\right)  \xi=\left(
\xi^{\delta}\cdot\nabla\right)  v^{\delta}-\left(  \xi\cdot\nabla\right)
v
\]
hence%
\[
\partial_{t}\left(  \xi^{\delta}-\xi\right)  +\left(  v^{\delta}%
\cdot\nabla\right)  \left(  \xi^{\delta}-\xi\right)  +\left(  \left(
v^{\delta}-v\right)  \cdot\nabla\right)  \xi=\left(  \xi^{\delta}%
\cdot\nabla\right)  \left(  v^{\delta}-v\right)  +\left(  \left(
\xi^{\delta}-\xi\right)  \cdot\nabla\right)  v.
\]
This implies, using \eqref{tril1}, that
\begin{align*}
\frac{1}{2}\frac{d}{dt}\left\Vert \xi^{\delta}-\xi\right\Vert _{L^{2}\left(
\mathbb{R}^{3}\right)  }^{2}   = & -\int\left(  \left(  v^{\delta
}-v\right)  \cdot\nabla\right)  \xi\cdot\left(  \xi^{\delta}-\xi\right)
dx\\
& +\int\left(  \xi^{\delta}\cdot\nabla\right)  \left(  v^{\delta
}-v\right)  \cdot\left(  \xi^{\delta}-\xi\right)  dx\\
& +\int\left(  \left(  \xi^{\delta}-\xi\right)  \cdot\nabla\right)
v\cdot\left(  \xi^{\delta}-\xi\right)  dx\\
\leq&\left\Vert D\xi\right\Vert _{L^{3}}\left\Vert v^{\delta}-v\right\Vert
_{L^{6}}\left\Vert \xi^{\delta}-\xi\right\Vert _{L^{2}}\\
& +\left\Vert \xi^{\delta}\right\Vert _{\infty}\left\Vert D\left(
v^{\delta}-v\right)  \right\Vert _{L^{2}}\left\Vert \xi^{\delta}%
-\xi\right\Vert _{L^{2}}\\
& +\left\Vert Dv\right\Vert _{\infty}\left\Vert \xi^{\delta}-\xi\right\Vert
_{L^{2}}^{2}.
\end{align*}
Since $H^{s}\left(  \mathbb{R}^{3}, \mathbb{R}^3 \right)  \subset W^{1,3}\left(
\mathbb{R}^{3}, \mathbb{R}^3\right)  $
\[
\left\Vert D\xi\right\Vert _{L^{3}}\leq C\left\Vert \xi\right\Vert _{H^{s}%
}\leq C.
\]
Moreover, since $H^{1}\left(  \mathbb{R}^{3}, \mathbb{R}^3\right)  \subset L^{6}\left(
\mathbb{R}^{3}, \mathbb{R}^3\right)  $
\begin{align*}
\left\Vert v^{\delta}-v\right\Vert _{L^{6}}  & \leq\left\Vert K^{\delta
}\ast\left(  \xi^{\delta}-\xi\right)  \right\Vert _{H^{1}}+\left\Vert
K^{\delta}\ast\xi-K\ast\xi\right\Vert _{H^{1}}\\
& =\left\Vert \rho^{\delta}\ast K\ast\left(  \xi^{\delta}-\xi\right)
\right\Vert _{H^{1}}+\left\Vert \rho^{\delta}\ast v-v\right\Vert _{H^{1}%
}\\
& \leq C\left\Vert K\ast\left(  \xi^{\delta}-\xi\right)  \right\Vert
_{H^{1}}+\left\Vert \rho^{\delta}\ast v-v\right\Vert _{H^{1}}\\
& \leq C\left\Vert \xi^{\delta}-\xi\right\Vert _{L^{2}}+\left\Vert
\rho^{\delta}\ast v-v\right\Vert _{H^{1}}%
\end{align*}%

In the same way we can estimate the term with the first order derivative,

\begin{align*}
\left\Vert D\left(  v^{\delta}-v\right)  \right\Vert _{L^{2}}  &
\leq\left\Vert DK^{\delta}\ast\left(  \xi^{\delta}-\xi\right)  \right\Vert
_{L^{2}}+\left\Vert DK^{\delta}\ast\xi-DK\ast\xi\right\Vert _{L^{2}}\\
& =\left\Vert \rho^{\delta}\ast DK\ast\left(  \xi^{\delta}-\xi\right)
\right\Vert _{L^{2}}+\left\Vert \rho^{\delta}\ast Dv-Dv\right\Vert
_{L^{2}}\\
& \leq C\left\Vert DK\ast\left(  \xi^{\delta}-\xi\right)  \right\Vert
_{L^{2}}+\left\Vert \rho^{\delta}\ast Dv-Dv\right\Vert _{L^{2}}\\
& \leq C\left\Vert \xi^{\delta}-\xi\right\Vert _{L^{2}}+\left\Vert
\rho^{\delta}\ast Dv-Dv\right\Vert _{L^{2}}.
\end{align*}
Now, using  $\left\Vert \xi^{\delta}\right\Vert
_{L^\infty}\leq C \left\Vert \xi^{\delta}\right\Vert
_{H^{s}}\leq C$, and 
$\left\Vert Dv\right\Vert _{L^\infty}\leq C$, we get that

\begin{align*}
\frac{1}{2}\frac{d}{dt}\left\Vert \xi^{\delta}-\xi\right\Vert _{L^{2}}^{2}\leq &
 C\left(  \left\Vert \xi^{\delta}-\xi\right\Vert _{L^{2}}+\left\Vert
\rho^{\delta}\ast v-v\right\Vert _{H^{1}}\right)  \left\Vert \xi
^{\delta}-\xi\right\Vert _{L^{2}}\\
& +C\left(  C\left\Vert \xi^{\delta}-\xi\right\Vert _{L^{2}}+\left\Vert
\rho^{\delta}\ast Dv-Dv\right\Vert _{L^{2}}\right)  \left\Vert
\xi^{\delta}-\xi\right\Vert _{L^{2}}\\
& +C\left\Vert \xi^{\delta}-\xi\right\Vert _{L^{2}}^{2}\\
\leq & C\left\Vert \xi^{\delta}-\xi\right\Vert _{L^{2}}^{2}+C\left\Vert
\rho^{\delta}\ast v-v\right\Vert _{H^{1}}^{2}+C\left\Vert \rho
^{\delta}\ast Dv-Dv\right\Vert _{L^{2}}^{2}.
\end{align*}%

Using Gronwall's lemma and the fact that $\left\Vert
\rho^{\delta}\ast v-v\right\Vert _{H^{1}}^{2}\rightarrow0$, $\left\Vert
\rho^{\delta}\ast Dv-Dv\right\Vert _{L^{2}}^{2}\rightarrow0$ uniformly in $t$ completes the proof.
\end{proof}

\begin{lemma}
\label{lemma inequality}Given $\delta,R>0$ there exists a constant
$C_{\delta,R}>0$ with the following property. If $\xi_{0},\widetilde{\xi}_{0}$
satisfy $\left\vert \xi_{0}\right\vert_{\mathcal{M}} \leq R$, $\left\vert \widetilde{\xi
}_{0}\right\vert_{\mathcal{M}} \leq R$, and $\xi_{t}^{\delta},\widetilde{\xi}_{t}^{\delta}$
are the corresponding solutions of equation \eqref{Kvorticity} on $\left[  0,T\right]  $,
then, globally in time we have
\begin{equation*}
\sup_{t\in\left[  0,T\right]  }\left\Vert \xi_{t}^{\delta}-\widetilde{\xi}%
_{t}^{\delta}\right\Vert \leq C_{\delta,R}\left\Vert \xi_{0}-\widetilde{\xi
}_{0}\right\Vert .\label{stima}%
\end{equation*}

\end{lemma}

\begin{proof}

Let $t\in [0,T]$. We have%
\begin{align*}
\left\vert \xi^\delta_{t}\left(  \theta\right)  -\widetilde{\xi_{t}^{\delta}}\left(  \theta\right)
\right\vert  &  =\left\vert \varphi_{\sharp}^{t,K^\delta\ast \xi}\xi_{0}\left(
\theta\right)  -\varphi_{\sharp}^{t,K^\delta\ast \widetilde{\xi}}\widetilde{\xi_{0}}\left(
\theta\right)  \right\vert \\
&  \leq\left\vert \varphi_{\sharp}^{t,K^\delta\ast \xi}\xi_{0}\left(  \theta\right)
-\varphi_{\sharp}^{t,K^\delta\ast \xi}\widetilde{\xi_{0}}\left(  \theta\right)  \right\vert
+\left\vert \varphi_{\sharp}^{t,K^\delta\ast \xi}\widetilde{\xi_{0}}\left(  \theta\right)
-\varphi_{\sharp}^{t,K^\delta\ast \widetilde{\xi}}\widetilde{\xi_{0}}\left(  \theta\right)  \right\vert
\\
&  =\left\vert \left(  \xi_{0}-\widetilde{\xi_{0}}\right)  \left(  \varphi_{\sharp
}^{t,K^\delta\ast \xi}\theta\right)  \right\vert +\left\vert \widetilde{\xi_{0}}\left(
\varphi_{\sharp}^{t,K^\delta\ast \xi}\theta-\varphi_{\sharp}^{t,K^\delta\ast \widetilde{\xi}}%
\theta\right)  \right\vert \\
&  \leq\left\Vert \xi_{0}-\widetilde{\xi_{0}}\right\Vert \left(  \Vert\varphi_{\sharp
}^{t,K^\delta\ast \xi}\theta\Vert_{\infty}+\text{Lip}(\varphi_{\sharp}^{t,K^\delta\ast \xi}%
\theta)\right)  +\left\vert \widetilde{\xi_{0}}\right\vert _{\mathcal{M}}\Vert
\varphi_{\sharp}^{t,K^\delta\ast \xi}\theta-\varphi_{\sharp}^{t,K^\delta\ast \widetilde{\xi}}\theta
\Vert_{\infty}.
\end{align*}

Now, from the definition of push-forward and \eqref{Dphi} we infer that
\begin{align*}
\Vert\varphi_{\sharp}^{t,K^\delta\ast \xi}\theta\Vert_{\infty}&=\Vert D\varphi^{t,K^\delta\ast \xi}(\cdot)^{T}
\theta(\varphi^{t,K^\delta\ast \xi}(\cdot))\Vert_{\infty}\leq\Vert
D\varphi^{t,K^\delta\ast \xi}\Vert_{\infty}\Vert\theta\Vert_{\infty}\leq e^{C^1_{\delta}T\left\Vert \xi_0\right\Vert}\|\theta\|_{\infty}\\
\end{align*}%
On the other hand, 
\begin{align*}
&\left | D\varphi^{t,K^\delta\ast \xi}(x)^{T}\theta(\varphi^{t,K^\delta\ast \xi}(x))- D\varphi^{t,K^\delta\ast \xi }(y)^{T}
\theta(\varphi^{t,K^\delta\ast \xi}(y)) \right | \\
&\leq \left | D\varphi^{t,K^\delta\ast \xi}(x)^{T}\theta(\varphi^{t,K^\delta\ast \xi}(x))- D\varphi^{t,K^\delta\ast \xi }(y)^{T}\theta(\varphi^{t,K^\delta\ast \xi }(x)) \right |\\
&+\left | D\varphi^{t,K^\delta\ast \xi }(y)^{T}\theta(\varphi^{t,K^\delta\ast \xi}(x))- D\varphi^{t,K^\delta\ast \xi }(y)^{T}
\theta(\varphi^{t,K^\delta\ast \xi}(y)) \right |\\
&\leq \|\theta\|_{\infty}  \left | D\varphi^{t,K^\delta\ast \xi }(x)- D\varphi^{t,K^\delta\ast \xi }(y) \right |+\text {Lip}(\theta) 
\|D\varphi^{t,K^\delta\ast \xi}\|_{\infty}|x-y|
\end{align*}
Hence, using \eqref{Dphi} and  \eqref{Dphi-Lip} , we infer that
\begin{align*}
\text{Lip}(\varphi_{\sharp}^{t,K^\delta\ast \xi }\theta)\leq TC_{\delta}^2\|\xi_0\| e^{2C^1_{\delta}T\|\xi_0\|}\|\theta\|_{\infty} +e^{C^1_{\delta}T\|\xi_0\|}\text{Lip}(\theta)
\end{align*}
Moreover,%
\begin{align*}
\Vert\varphi_{\sharp}^{t,K^\delta\ast \xi}\theta-\varphi_{\sharp}^{t,K^\delta\ast \widetilde{\xi}}%
\theta\Vert_{\infty}  &  =\Vert D\varphi^{t,K^\delta\ast \xi }(\cdot)^{T}\theta
(\varphi^{t, K^\delta\ast \xi }(\cdot))-D\varphi^{t,B(\widetilde{\xi})}(\cdot)^{T}\theta
(\varphi^{t,K^\delta\ast \widetilde{\xi}}(\cdot))\Vert_{\infty}\\
&  \leq\Vert D\varphi^{t,K^\delta\ast \xi }(\cdot)^{T}\theta(\varphi^{t,K^\delta\ast \xi}%
(\cdot))-D\varphi^{t, K^\delta\ast \xi }(\cdot)^{T}\theta(\varphi^{t,K^\delta\ast \widetilde{\xi}}%
(\cdot))\Vert_{\infty}\\
&  +\Vert D\varphi^{t,K^\delta\ast \xi}(\cdot)^{T}\theta(\varphi^{t,K^\delta\ast \widetilde{\xi}}%
(\cdot))-D\varphi^{t,B(\widetilde{\xi})}(\cdot)^{T}\theta(\varphi^{t,K^\delta\ast \widetilde{\xi}}%
(\cdot))\Vert_{\infty}\\
  &  \leq\Vert D\varphi^{t,K^\delta\ast \xi}\Vert_{\infty}\text{Lip}(\theta)\Vert
\varphi^{t,K^\delta\ast \xi}-\varphi^{t,K^\delta\ast \widetilde{\xi}}\Vert_{\infty}+\Vert D\varphi
^{t,K^\delta\ast \xi}-D\varphi^{t,K^\delta\ast \widetilde{\xi}}\Vert_{\infty}\Vert\theta\Vert_{\infty}\\
\end{align*}

Using \eqref{Dphi},  \eqref{Lip-Dphi}  and \eqref{Lip-phi} we get that

\begin{align*}
\Vert\varphi_{\sharp}^{t,K^\delta\ast \xi }\theta-\varphi_{\sharp}^{t,K^\delta\ast \widetilde{\xi}}%
\theta\Vert_{\infty}& \leq   \text{Lip}(\theta) C^0_{\delta}C^1_{\delta} \|\xi_0\| 
e^{2T C^1_{\delta} \Vert \xi_0\Vert}\int_{0}^{T}\left\Vert \xi_s-\widetilde{\xi}_s\right\Vert ds\\
&+\Vert\theta\Vert_{\infty}
C^2_{\delta}\|\xi_0\| e^{C^1_{\delta}T  (\left\Vert \xi_0\right\Vert
+\left\Vert \widetilde{\xi}_0\right\Vert )}\left(  1+C^0_K C^1_{\delta}T \left\Vert
\xi_0\right\Vert e^{C^1_{\delta}T\left\Vert \xi_0\right\Vert }\right)
\int_{0}^{T}\Vert\xi_s-\widetilde{\xi_s}\Vert ds
\end{align*}

Collecting all these estimates, we get that

\begin{align*}
\|\xi^\delta_{t}-\widetilde{\xi^\delta_{t}}\|&\leq \|\xi_{0}-\widetilde{\xi_{0}}\|
\left(TC_{\delta}^2\|\xi_0\| +2\right)e^{2C^1_{\delta}T\|\xi_0\|} 
+\left(\left\vert \widetilde{\xi_{0}}\right\vert _{\mathcal{M}}\|\xi_0\|\int_{0}^{T}\Vert\xi^\delta_s-\widetilde{\xi^\delta_s}\Vert ds\right)\times\\
&\left( C^0_{\delta}C^1_{\delta} e^{2T C^1_{\delta} \Vert \xi_0\Vert} + C^2_{\delta}e^{C^1_{\delta}T  (\left\Vert \xi_0\right\Vert
+\left\Vert \widetilde{\xi}_0\right\Vert )}\left(  1+C^0_K C^1_{\delta}T \left\Vert
\xi_0\right\Vert e^{C^1_{\delta}T\left\Vert \xi_0\right\Vert }\right) \right)
\end{align*}
Using $\xi_{0}$ and $\widetilde{\xi_{0}}$ we get that
\begin{align*}
\|\xi^\delta_{t}-\widetilde{\xi^\delta_{t}}\|&\leq \|\xi_{0}-\widetilde{\xi_{0}}\|
\left(TC_{\delta}^2 R +2\right)e^{2C^1_{\delta}T R } 
+\left(R^{2} \int_{0}^{T}\Vert\xi_s^\delta-\widetilde{\xi_s^\delta}\Vert ds\right) \times\\
&\left( C^0_{\delta}C^1_{\delta} + C^2_{\delta}
\left(  1+C^0_K C^1_{\delta}TR e^{C^1_{\delta}TR  }\right)\right)e^{2T C^1_{\delta} R }
\end{align*}
Set 
\begin{equation*}
C^{\ast}:=R^{2}\left( C^0_{\delta}C^1_{\delta} + C^2_{\delta}
\left(  1+C^0_K C^1_{\delta}TR e^{C^1_{\delta}TR  }\right)\right)e^{2T C^1_{\delta} R }
\end{equation*}
and 
\begin{equation*}
C_{\ast}:=\left(TC_{\delta}^2 R +2\right)e^{2C^1_{\delta}T R } 
\end{equation*}
Hence, using Gronwall lemma we deduce that
\begin{equation*}\label{Q}
\sup_{t\in\left[  0,T\right]  }\|\xi^\delta_{t}-\widetilde{\xi^\delta_{t}}\| \leq C_{\ast} \|\xi_{0}-\widetilde{\xi_{0}}\| e^{TC^{\ast}  } 
\end{equation*}
and this completes the proof by setting  $C_{\delta,R}:= C_{\ast} e^{TC^{\ast}}$.
\end{proof}

\subsection{Mean field result}\label{section final convergence}

\begin{theorem} \label{main_theorem}
	
Let $\xi_{0}^{N}$ be a sequence of currents and $\xi_{0}\in H^{s}$, $s>3/2$,
such that 
\begin{equation*}
\lim_{N\rightarrow\infty}\left\Vert \xi_{0}^{N}-\xi_{0}\right\Vert
=0.
\end{equation*}
Let $\xi_{t}$ be the solution in $H^{s}$ of equation \eqref{vorticity} on the interval
$\left[  0,T_{0}\right]  $ with initial condition $\xi_{0}$. For every
$\delta>0$ and positive integer $N$, let $\xi_{t}^{N,\delta}$ satisfying  \eqref{xi-N0} and \eqref{interacting curves} 
 on the interval $\left[  0,T\right]  $ with initial condition
$\xi_{0}^{N}$. Let $R>0$ be such that $\left\vert \xi_{0}\right\vert_{\mathcal{M}} \leq R$,
$\left\vert \xi_{0}^{N}\right\vert_{\mathcal{M}} \leq R$, and let $C_{\delta,R}>0$ be the
corresponding constant of Lemma \ref{lemma inequality}. Let $\delta
_{N}\rightarrow0$ be a sequence. If
\begin{equation}\label{initial condition convergence}
\lim_{N\rightarrow\infty}C_{\delta_{N},R}\left\Vert \xi_{0}^{N}-\xi
_{0}\right\Vert =0
\end{equation}
then%
\[
\lim_{N\rightarrow\infty}\sup_{t\in\left[  0,T\right]  }\left\Vert \xi
_{t}^{N,\delta_{N}}-\xi_{t}\right\Vert =0.
\]

\end{theorem}

\begin{proof}
Denote by $\xi_{t}^{\delta}$ the solution on $\left[  0,T\right]  $ of
equation \eqref{PDE for currents} with initial condition $\xi_{0}$. We have%
\begin{align*}
\left\Vert \xi_{t}^{N,\delta_{N}}-\xi_{t}\right\Vert  & \leq\left\Vert \xi
_{t}^{N,\delta_{N}}-\xi_{t}^{\delta_{N}}\right\Vert +\left\Vert \xi
_{t}^{\delta_{N}}-\xi_{t}\right\Vert \\
& \leq C_{\delta_{N},R}\left\Vert \xi_{0}^{N}-\xi_{0}\right\Vert +\left\Vert
\xi_{t}^{\delta_{N}}-\xi_{t}\right\Vert .
\end{align*}
Recall that $\lim_{N\rightarrow\infty}\sup_{t\in\left[  0,T\right]
}\left\Vert \xi_{t}^{\delta_{N}}-\xi_{t}\right\Vert =0$, from Theorem \ref{delta-cv}.\\
Then $\lim_{N\rightarrow\infty}\sup_{t\in\left[  0,T\right]
}\left\Vert \xi_{t}^{N,\delta_{N}}-\xi_{t}\right\Vert =0$.
\end{proof}

\vspace{1cm}

\noindent\textbf{Acknowledgements}:  Hakima Bessaih's research is
partially supported by NSF grant DMS-1418838.

\end{document}